\documentclass[oneside,reqno]{amsart}
\usepackage{etex}
\usepackage[a4paper]{geometry}
\usepackage{amssymb,amsfonts,amsmath}
\usepackage{comment} 
\usepackage[all,arc]{xy}
\usepackage{enumerate}
\usepackage{mathrsfs,mathtools}
\usepackage{todonotes,booktabs}
\usepackage{stmaryrd}
\usepackage{marvosym}
\usepackage{graphicx}
\usepackage{pdflscape}

\usepackage{bbm}
\usepackage{tikz}
\usepackage{tikz-cd}
\usetikzlibrary{trees}
\usetikzlibrary[shapes]
\usetikzlibrary[arrows]
\usetikzlibrary{patterns}
\usetikzlibrary{fadings}
\usetikzlibrary{backgrounds}
\usetikzlibrary{decorations.pathreplacing}
\usetikzlibrary{decorations.pathmorphing}
\usetikzlibrary{positioning}
	
\usepackage{xcolor} 
\usepackage{graphicx}

\usepackage[pagebackref]{hyperref}

\definecolor{dark-red}{rgb}{0.5,0.15,0.15}
\definecolor{dark-blue}{rgb}{0.15,0.15,0.6}
\definecolor{dark-green}{rgb}{0.15,0.6,0.15}
\hypersetup{
    colorlinks, linkcolor=dark-red,
    citecolor=dark-blue, urlcolor=dark-green
}

\newcommand{\stc}{S^3 \langle 3 \rangle}

\newcommand{\cX}{\cal{X}}

\renewcommand*{\backref}[1]{}
\renewcommand*{\backrefalt}[4]{%
  \ifcase #1 %
No citations.
  \or
(cit. on p. #2).%
  \else
(cit on pp. #2).%
  \fi%
}
\usepackage{subfiles}

\usepackage[nameinlink,capitalise,noabbrev]{cleveref}

\newtheorem{thm}{Theorem}[section]
\newtheorem{cor}[thm]{Corollary}
\newtheorem{prop}[thm]{Proposition}
\newtheorem{lem}[thm]{Lemma}

\newtheorem*{thm*}{Theorem}

\theoremstyle{definition}
\newtheorem{defn}[thm]{Definition}

\newtheorem{ex}[thm]{Example}

\theoremstyle{remark}
\newtheorem{rem}[thm]{Remark}

\makeatletter
\let\c@equation\c@thm
\makeatother
\numberwithin{equation}{section}

\DeclareMathOperator{\Hom}{Hom}

\DeclareMathOperator{\colim}{colim}
\DeclareMathOperator{\cA}{\mathcal{A}}
\DeclareMathOperator{\cC}{\mathcal{C}}

\DeclareMathOperator{\rank}{rank}

\DeclareMathOperator{\cF}{\mathcal{F}}
\DeclareMathOperator{\cG}{\mathcal{G}}

\DeclareMathOperator{\Rep}{Rep}
\DeclareMathOperator{\Syl}{Syl}

\DeclareMathOperator{\TrD}{\mathbf{TrDeg}}

\DeclareMathOperator{\depth}{depth}

\DeclareMathOperator{\res}{res}

\DeclareMathOperator{\Tr}{Tr}

\usepackage[shortlabels]{enumitem}

\newcommand{\bA}{\mathbf{A}}
\DeclareMathOperator{\supp}{supp}

\newcommand{\cK}{\mathcal{K}}

\newcommand{\cal}{\mathcal}
\newcommand{\xr}{\xrightarrow}

\newcommand{\Z}{\mathbb{Z}}

\Crefname{figure}{Figure}{Figures}
\Crefname{assu}{Assumption}{Assumptions}
\Crefname{lem}{Lemma}{Lemmas}
\Crefname{thm}{Theorem}{Theorems}
\Crefname{ex}{Example}{Examples}
\Crefname{prop}{Proposition}{Propositions}

\renewcommand{\frak}{\mathfrak}

\DeclareMathOperator{\Inj}{Inj}

\newcommand{\recollement}[5]{
\xymatrix{{#1} \ar[r]|-{#2} & #3 \ar[r]|-{#4} \ar@<1ex>[l]^-{{#2}_!} \ar@<-1ex>[l]_-{{#2}^*} & #5, \ar@<1ex>[l]^-{{#4}!} \ar@<-1ex>[l]_-{{#4}^*}
}}
\let\lim\relax

\DeclareMathOperator{\lim}{lim}

\newcommand{\cU}{\mathcal{U}}

\newcommand{\F}{\mathbb{F}}

\usepackage{extarrows}

\newcommand{\cL}{\mathcal{L}}

\DeclareMathOperator{\Map}{Map}
\title{Depth and detection for Noetherian unstable algebras}

\author{Drew Heard}
\address{Fakult{\"a}t f{\"u}r Mathematik, Universit{\"a}t Regensburg}
  \email{drew.k.heard@gmail.com}
\thanks{Supported by SFB 1085 'Higher Invariants' (Universit{\"a}t Regensburg), funded by the Deutsche Forschungsgemeinschaft.}

\date{\today}

\begin{document}

\begin{abstract}
For a connected Noetherian unstable algebra $R$ over the mod $p$ Steenrod algebra, we prove versions of theorems of Duflot and Carlson on the depth of $R$, originally proved when $R$ is the mod $p$ cohomology ring of a finite group.  This recovers the aforementioned results, and also proves versions of them when $R$ is the mod $p$ cohomology ring of a compact Lie group, a profinite group with Noetherian cohomology, a Kac--Moody group, a discrete group of finite virtual cohomological dimension, as well as for certain other discrete groups. More generally, our results apply to certain finitely generated unstable $R$-modules. Moreover, we explain the results in the case of the $p$-local compact groups of Broto, Levi, and Oliver, as well as in the modular invariant theory of finite groups. 
\end{abstract}

\maketitle
\setcounter{tocdepth}{1}
\tableofcontents
\section{Introduction}
For a compact Lie group $G$, the mod $p$ cohomology ring, which we denote $H_G^*$, is a  finitely generated graded-commutative $\F_p$-algebra, or equivalently a graded-commutative Noetherian ring. Computing this group cohomology can be exceedingly difficult. However, Quillen \cite{Quillen1971spectrum} showed that the cohomology ring could be approximated, in a certain sense, by the cohomology of it its elementary abelian $p$-subgroups. Using this, he showed that the Krull dimension of $H_G^*$ is equal to the $p$-rank of $G$, as had been conjectured by Atiyah and Swan. 

One can also ask about the depth of $H_G^*$, namely the maximal length of a regular sequence of homogeneous elements in $H_G^{> 0}$, the maximal ideal of positive degree elements in $H_G^*$ (see \Cref{sec:depth_algebra} for the precise definition of depth). Finding a group-theoretic description of the depth of $H_G^*$ is a difficult problem. Since the depth is always less than the Krull dimension, the depth is bounded above by the $p$-rank of $G$. Standard commutative algebra also gives the improved bound that the depth is bounded above by the minimum dimension of an associated prime of $H_G^*$, i.e., the minimum of the Krull dimensions of $H_G^*/\frak{p}$, where $\frak p$ runs through the associated primes of $G$. The first minimum bound was given by Duflot \cite{Duflot1981Depth}. 
\begin{thm}[Duflot]\label{thm:duflot_original}
  Let $G$ be a finite group, then the depth of $H_G^*$ is greater than or equal to the rank of the maximal central elementary abelian $p$-subgroup of $G$.
\end{thm}
In fact, Duflot's considered more generally the depth of the $H_G^*$-module $H_G^*(X)$, where $X$ is a finite $G$-CW complex. If $G$ has Sylow $p$-subgroup $S$, then one can improve on Duflot's theorem slightly: the depth of $H_G^*$ is greater than or equal to the rank of the maximal central elementary abelian $p$-subgroup of $S$. An alternative proof of \Cref{thm:duflot_original}, which also works in the case of a compact Lie group, was given by Broto and Henn \cite{BrotoHenn1993Some}, and exploits the fact that $H_G^*$ is a $H_C^*$-comodule, via the multiplication map $C \times G \to G$ (which is a group homomorphism). Later, in the case $G$ finite, Carlson \cite[Proposition 5.2]{Carlson1999Problems} showed that if $x_1,\ldots,x_n$ is a sequence of homogeneous elements of $H_G^*$ whose restriction to the center of a Sylow $p$-subgroup $S$ of $G$ is a regular sequence, then $x_1,\ldots,x_n$ is a regular sequence in $H_G^*$. One easily recovers Duflot's theorem from this result. 

The $\F_p$-algebra $H_G^*$ is an example of an unstable algebra over the Steenrod algebra \cite{schwartz_book}. One can ask more generally about the depth of an arbitrary Noetherian unstable algebra. A deep result along these lines is a result of Bourguiba and Zarati \cite{CampbellWehlau2011Modular}, which gives the depth of $R$ in terms of the Dickson invariants, settling the Landweber--Stong conjecture.

We will give a version that is closer in spirit to Duflot's theorem. In order to describe this, we need to explain what the analog of a central subgroup is. For this, we follow Dwyer and Wilkerson \cite{DwyerWilkerson1992cohomology}, and define centrality in terms of Lannes' $T$-functor. Briefly, let $R$ be a connected Noetherian unstable algebra, and $g \colon R \to H_C^*$ a morphism of unstable algebras making $H_C^*$ into a finitely generated $R$-module, where $C$ is an elementary abelian $p$-group.\footnote{Groups denoted $C$ and $E$ will always be elementary abelian $p$-groups throughout this paper.} Such pairs $(C,g)$ form the objects of a category $\bA_R$. For each pair $(C,g) \in \bA_R$ there is an unstable algebra $T_C(R;g)$, which is a component of Lannes' $T$-functor $T_CR$, and a canonical map $\rho_{R,(C,g)} \colon R \to T_C(R;g)$. We say that the pair $(C,g)$ is central if this map is an isomorphism. The following is a special case of our first main theorem. 
\begin{thm*}[\Cref{thm:duflot_regular,cor:duflot}]
  Let $R$ be a connected Noetherian unstable algebra and $(C,g)$ a non-trivial central object in $\bA_R$. If $x_1,\ldots,x_n$ is a sequence of homogeneous elements in $R$ such that $x_1,\ldots,x_n$ form a regular sequence in $H_C^*$ considered as an $R$-module via $g$, then $x_1,\ldots,x_n$ is a regular sequence in $R$. In particular, $\depth(R) \ge \rank(C)$. 
  \end{thm*} 
  The proof follows the argument of Broto--Henn and Carlson, and uses the fact that if $(C,g)$ is central in $\bA_R$, then $R$ has the structure of a $H_C^*$-comodule, see \Cref{prop:comdule_central}. In \Cref{sec:examples}, we explain how this recovers Duflot's result for compact Lie groups (in fact, even for the Borel equivariant cohomology of finite $G$-CW complexes), and also works for a larger class of groups, such as groups of finite virtual cohomological dimension, profinite groups with finitely generated $\F_p$-cohomology, and Kac--Moody groups.

   We also explain the result in the case of saturated fusion systems on finite $p$-groups, or more generally discrete $p$-toral groups (these are the $p$-local compact groups of Broto--Levi--Oliver \cite{BrotoLeviOliver2003homotopy,BrotoLeviOliver2007Discrete}). We refer the reader to \Cref{sec:plocal} for more details about $p$-local compact groups. For now, we simply remind the reader that a discrete $p$-toral group is a group which contains a normal subgroup $S_0 \cong (\Z/p^{\infty})^r$ for some $r \ge 0$ such that $S/S_0$ is a finite $p$-group. Such groups always have a non-trivial center. We will write $H_{\cF}^*$ for the cohomology of the saturated fusion system $\cal{F}$. Finally, we note that  if $\cF$ is a saturated fusion system on a discrete $p$-toral group $S$, then there is a canonical restriction map $H_{\cF}^* \to H_S^*$, and hence given any elementary abelian subgroup $E \le S$, there is an induced map $H_{\cF}^* \to H_E^*$. 
  \begin{thm*}[\Cref{thm:depthplocalcompact}]
  Let $\cF$ be a fusion system on a discrete $p$-toral group $S$, and let $C$ denote a central elementary abelian $p$-subgroup of $S$. If $x_1,\ldots,x_n$ is a sequence of homogeneous elements in $H_{\cF}^*$ such that the restriction of $x_1,\ldots,x_n$ form a regular sequence in $H_C^*$, then $x_1,\ldots,x_n$ form a regular sequence in $H_{\cF}^*$. In particular, the depth of $H_{\cF}^*$ is positive, and is greater than or equal to the $p$-rank of the center of $S$. 
  \end{thm*} 
  In particular, this is a version of Duflot's theorem for $p$-compact groups, or more generally finite loop spaces. We finish the examples, by explaining some implications in modular invariant theory in \Cref{sec:modular}.  

  Our second main theorem is originally due to Carlson in the special where $R = H_G^*$ is the cohomology of a finite group \cite{Carlson1995Depth}. We state it in the form given in \cite[Corollary 12.5.3]{CarlsonTownsleyValeriElizondoZhang2003Cohomology}. As always, $E$ will denote an elementary abelian $p$-group. 
  \begin{thm}[Carlson]
    Let $G$ be a finite group, and suppose that $H_G^*$ has depth $s$, then the product of restriction maps
    \[
\xymatrix{H_G^* \ar[r] & \displaystyle\prod_{\stackrel{E < G}{\rank(E) = s}} H_{C_G(E)}^*}
    \]
    is injective.
  \end{thm}
  The proof relies on a result of Benson regarding the image of the transfer $\Tr_K^G \colon H_K^* \to H_G^*$ for a subgroup $K < G$. An extension to the case of compact Lie groups was given by Cameron \cite[Theorem 4.13]{cameron_duflot} using another theorem of Duflot. Both proofs do not seem easy to generalize to an arbitrary unstable algebra. Rather, we use a deep result due to Henn, Lannes, and Schwartz \cite{HennLannesSchwartz1995Localizations}, as well as some $T$-functor technology. That such a proof should be possible is already mentioned in \cite{Carlson1995Depth}, and a (different) proof was given in the unpublished master's thesis of \cite{Poulsen}. 

  Our theorem in fact works very generally. Namely, let $R$ be a connected Noetherian unstable algebra, and $M$ an unstable module that is compatibly a finitely generated $R$-module (see \Cref{sec:tfunctorcomponent} for the exact details). Such objects form the objects of a category $R_{fg}-\cU$. Then, given a pair $(E,f) \in \bA_R$, there is a natural map $\rho_{M,(E,f)} \colon M \to T_E(M;f)$. Let $\bA_R^s$ denote the full subcategory of $\bA_R$ with objects pairs $(E,f)$ where $E$ is an elementary abelian $p$-subgroup of rank $s$. The maps above then assemble to give a morphism
   \[
\xymatrix{
  \phi_s \colon M \ar[r] & \displaystyle \prod_{(E,f) \in \bA_R^s} T_E(M;f).
}
   \]
   Our generalization of Carlson's theorem is the following. 
  \begin{thm*}(\Cref{thm:carlson})
   Let $R$ be a Noetherian unstable algebra, and $M \in R_{fg}-\cU$. Suppose that $\depth_R(M) \ge s$, then
   \[
\xymatrix{
  \phi_s \colon M \ar[r] & \displaystyle \prod_{(E,f) \in \bA_R^s} T_E(M;f)
}
   \]
   is injective. 
    \end{thm*}
    For example, we get the following group theoretic results. 
    \begin{thm*}(\Cref{thm:carlson_borel})
      Suppose we are in one of the following cases:
      \begin{enumerate}
          \item $G$ is a compact Lie group, and $X$ is a $G$-$CW$-complex with finitely many $G$-cells. 
    \item $G$ is a discrete group for which there exists a mod $p$ acyclic $G$-CW complex with finitely many $G$-cells and finite isotropy groups, and $X$ is any $G$-$CW$ complex with finitely many $G$-cells and with finite isotropy groups. 
      \end{enumerate}
      If $\depth_{H_G^*}(H_G^*(X)) \ge s$, then the product of restriction maps 
         \[
\xymatrix{
  H_G^*(X) \ar[r] & \displaystyle \prod_{\stackrel{E < S}{\rank(E) = s}} H_{C_G{(E)}}^*(X^E)
}
   \]
   is injective. 
    \end{thm*}
In the case where $X$ is a point, the result applies to a larger class of groups. 
    \begin{thm*}(\Cref{thm:carlson_group2})
    Suppose we are in one of the following cases:
    \begin{enumerate}
    \item $G$ is a profinite group such that the continuous mod $p$ cohomology $H_G^*$ is finitely generated as an $\F_p$-algebra. 
    \item $G$ is a discrete group of finite virtual cohomological dimension. 
    \item $G$ is a Kac--Moody group.
    \end{enumerate}
          If $\depth(H_G^*) \ge s$, then the product of restriction maps 
         \[
\xymatrix{
H_G^* \ar[r] & \displaystyle \prod_{\stackrel{E < S}{\rank(E) = s}} H_{C_G{(E)}}^*
}
   \]
   is injective.
    \end{thm*}
    Examples of such groups include $p$-adic analytic Lie groups, $(S)$-arithmetic groups, mapping class groups of orientable surfaces, outer automorphisms of free groups, and the word-hyperbolic groups of Gromov (see the discussion on references on page 192 of \cite{Henn1996Commutative}). 
    We also explain the results in the case of saturated fusion system on discrete $p$-toral groups. Here we can construct centralizer fusion systems, and the result is analogous to the group theoretic result above. 

\subsection*{Organization}    The paper is organized as follows. In \Cref{sec:centers} we review Lannes' $T$-functor, and centers of unstable algebras. In \Cref{sec:depth} we prove our versions of the theorems of Duflot and Carlson, while we finish with many examples in \Cref{sec:examples}. 
\subsection*{Acknowledgements} We thank Geoffrey Powell for a helpful conversation regarding \cite{powell}. 
\section{Centers of unstable algebras}\label{sec:centers}
\subsection{Unstable algebras and Lannes' \texorpdfstring{$T$}{T}-functor}\label{sec:tfunctor}
We begin with a brief review of unstable algebras over the Steenrod algebra, and Lannes' $T$-functor. General references for this material include the papers of Lannes \cite{lannes_unpublished,lannes_ihes}, the book of Schwartz \cite{schwartz_book}, or the notes of Henn \cite{henn_notes}. 

We let $\cU$ denote the category of unstable modules over the mod $p$ Steenrod algebra, and $\cK$ the category of unstable algebras over the mod $p$ Steenrod algebra. A typical object of $\cK$ is the mod $p$ cohomology of a space $X$ for some prime $p$. We note that in general the cohomology $H^*(X)$ is a graded-commutative ring, and we assume the same for any unstable algebra $R$. We say that $R$ is connected if $R^0 \cong \F_p$ (note that $R^i = 0$ for $i < 0$ is forced by the assumption that $R \in \cK$). 

Let $E$ be an elementary abelian $p$-group. We recall that Lannes' $T$-functor $T_E$ is left adjoint to tensoring with $H_E^*$ in the category $\cal{U}$, i.e., there is an isomorphism
\[
\Hom_{\cU}(T_ER,S) \cong \Hom_{\cU}(R,H_E^* \otimes S)
\]
for $R,S \in \cU$. The functor $T_E$ is exact and commutes with tensor products, and restricts to a functor $\cK \to \cK$, see \cite{lannes_ihes} or \cite[Theorem 3.2.2 and Theorem 3.8.1]{schwartz_book}.

There are several maps that we will use repeatedly throughout this paper. Let $M$ be an unstable algebra, then the adjoint of the identity map $T_EM \to T_EM$ gives rise a morphism $\eta_{M,E} \colon M \to H_E^* \otimes T_EM$ (this is the counit of the adjunction). We define a map $\kappa_{M,E} \colon T_EM \to H_E^* \otimes T_EM$ to be the adjoint of the composite
\[
\xymatrix{M \ar[r]^-{\eta_{M,E}}  & H_E^* \otimes T_EM \ar[r]^-{\Delta \otimes 1} & H_E^* \otimes H_E^* \otimes T_EM, }
\]
where $\Delta$ is the comultiplication map $H_E^* \to H_E^* \otimes H_E^*$. As shown in \cite[Section 1.13]{HennLannesSchwartz1995Localizations} for each $E$, the map $\kappa_{M,E}$ gives $T_EM$ the structure of a $H_E^*$-comodule. In fact, for the purposes of the paper we will only need to know that the composite
\begin{equation}\label{eq:comodidentity}
\xymatrix{T_EM \ar[r]^-{\kappa_{M,E}} & H_E^* \otimes T_EM \ar[r]^-{\epsilon_E \otimes 1} & T_EM}
\end{equation}
is the identity, where $\epsilon_E \colon H_E^* \to \F_p$ is the canonical projection. But the adjoint to this is the morphism $M \to H_E^* \otimes T_EM$ given by the composite $(1 \otimes \epsilon_E \otimes 1) \circ (\Delta \otimes 1) \circ \eta_{M,E} \simeq \eta_{M,E}$, and hence the composite is indeed the identity. 

Finally, we observe that given a morphism $\alpha \colon E \to V$ of elementary abelian $p$-groups, to give a map $T_EM \to T_VM$ we can give its adjoint, namely a morphism $M \to H_E^* \otimes T_VM$; there is an obvious candidate, namely the composite
\[
\xymatrix{M \ar[r]^-{\eta_{M,V}} & H_V^* \otimes T_VM \ar[r]^-{\alpha^* \otimes 1} & H_E^* \otimes T_VM}
\]
In the particular case where $\alpha \colon \{ e\} \to E$ is the inclusion of the trivial group, we obtain a map $\rho_{M,E} \colon T_{\{e \}}R \cong R \to T_ER$, which is equivalently given by the composite 
\[
\xymatrix{M \ar[r]^-{\eta_{M,E}} & H_E^* \otimes T_EM \ar[r]^-{\epsilon_E \otimes 1} &  T_EM}.
\]

Given $f \in \Hom_{\cK}(T_EM,S)$, let us write $f^{\#}$ for the corresponding adjoint map $M \to H_E^* \otimes S$. By taking the adjoint of the composite $T_EM \xlongequal{\text{id}} T_EM \xr{f} S$, we see that $f^{\#}$ is given by the composite 
\begin{equation}\label{eq:adjoint}
\xymatrix{M \ar[r]^-{\eta_{M,E}} & H_E^* \otimes T_EM \ar[r]^-{1 \otimes f} & H_E^* \otimes S}.
\end{equation}
This gives the following result. 
\begin{lem}\label{lem:commdia}
	For any map $f \colon T_EM \to S$ the diagram
\[
\begin{tikzcd}
	M \arrow{r}{\rho_{M,E}} \arrow[swap]{d}{f^{\#}} & T_EM \arrow{d}{f} \\
	H_E^* \otimes S \arrow[swap]{r}{\epsilon_E \otimes 1} & S,
\end{tikzcd}
\]
commutes. 
\end{lem}
\begin{proof}
	As noted, $f^{\#}$ factors as the composite $(1 \otimes f)\circ \eta_{M,E}$. It follows that 
\[
\begin{split}
(\epsilon_E \otimes 1)\circ f^{\#} &\cong (\epsilon_E \otimes 1)\circ (1 \otimes f)\circ \eta_{M,E}\\
&\cong   f \circ (\epsilon_E \otimes 1)\circ \eta_{M,E}\\
& \cong f \circ \rho_{M,E}
\end{split}
\]
as required. 
\end{proof}
We deduce the following. 
\begin{cor}\label{lem:comm}
	There is an isomorphism $ \kappa_{M,E} \circ \rho_{M,E} \cong \eta_{M,E}$. 
\end{cor}
\begin{proof}
	By \Cref{lem:commdia} there is a commutative diagram
	\[
\begin{tikzcd}[column sep=0.75in]
	M \arrow{r}{\rho_{M,E}} \arrow[swap]{d}{\kappa_{M,E}^{\#}} & T_EM \arrow{d}{\kappa_{M,E}} \\
	H_E^* \otimes H_E^* \otimes T_EM \arrow[swap]{r}{\epsilon_{E} \otimes 1 \otimes 1} & H_E^* \otimes T_EM,
\end{tikzcd}
\]
But by definition $\kappa_{M,E}^{\#} = (\Delta \otimes 1)\circ \eta_{M,E}$. Since $(\epsilon_E \otimes 1)\circ \Delta \colon H_E^* \to H_E^*$ is the identity, we deduce that $ \kappa_{M,E} \circ \rho_{M,E} \cong \eta_{M,E}$, as claimed. 
\end{proof}
\begin{rem}
By \eqref{eq:adjoint} we also deduce that $\kappa_{M,E}^{\#} \cong (1 \otimes \kappa_{M,E}) \circ \eta_{M,E}$ and hence $(\Delta \otimes 1) \circ \eta_{M,E} \cong (1 \otimes \kappa_{M,E}) \circ \eta_{M,E}$. This can be used to show the coassociativity of the $H_E^*$-comodule structure on $T_EM$, i.e., that the diagram
\[
\begin{tikzcd}
T_EM \arrow[r, "{\kappa_{M,E}}"] \arrow[d, "{\kappa_{M,E}}"'] & H_E^* \otimes T_EM \arrow[d, "{1 \otimes \kappa_{M,E}}"] \\
H_E^* \otimes T_EM \arrow[r, "\Delta \otimes 1"']             & H_E^* \otimes H_E^* \otimes T_EM                    
\end{tikzcd}
\] commutes. Indeed, taking adjoints, it suffices to show that the diagram	
	\[
	\begin{tikzcd}
M \arrow[r, "{\eta_{M,E}}"] \arrow[d, "{\eta_{M,E}}"'] & H_E^* \otimes T_EM \arrow[r, "\Delta \otimes 1"]                         & H_E^* \otimes H_E^* \otimes T_EM \arrow[d, "{1 \otimes 1 \otimes \kappa_{M,E}}"] \\
H_E^* \otimes T_EM \arrow[r, "\Delta \otimes 1"']      & H_E^* \otimes H_E^* \otimes T_EM \arrow[r, "1 \otimes \Delta \otimes 1"'] & H_E^* \otimes H_E^* \otimes H_E^* \otimes T_EM
\end{tikzcd}
	\]
	commutes. To see this, first observe that $(1 \otimes \Delta \otimes 1)\circ (\Delta \otimes 1) \cong (\Delta \otimes 1 \otimes 1) \circ (\Delta \otimes 1)$. We then have
	\[
	\begin{split}
(1 \otimes \Delta \otimes 1) \circ (\Delta \otimes 1) \circ \eta_{M,E} &\cong (\Delta \otimes 1 \otimes 1) \circ (\Delta \otimes 1) \circ \eta_{M,E} \\
&\cong (\Delta \otimes 1 \otimes 1)  \circ (1 \otimes \kappa_{M,E}) \circ \eta_{M,E}\\
& \cong (1 \otimes 1 \otimes \kappa_{M,E}) \circ (\Delta \otimes 1) \circ \eta_{M,E}
\end{split}
	\]
	as required.
\end{rem}
\subsection{Components of the \texorpdfstring{$T$}{T}-functor}\label{sec:tfunctorcomponent}
Given an unstable algebra $R$, we can define a category $R-\cU$, whose objects are unstable $\cA$-modules $M$ together with $\cal{A}$-linear structure maps $R \otimes M \to M$ which make $M$ into an $R$-module, and whose morphisms are the $\cal{A}$-linear maps which are also $R$-linear. The full subcategory consisting of the finitely generated $R$-modules will be denoted $R_{fg}-\cU$. 

For any unstable algebra $R$, and $\cal{K}$-morphism $f \colon R \to H_E^*$, we define $T_E(R;f)$ as the tensor product $T_E(R) \otimes_{T^0_E(R)} \F_p(f)$, where $\F_p(f)$ denotes $\F_p$ with $T_E(R)$ module structure induced by the adjoint of the map $f \colon R \to H_E^*$.

We note that $T^0_E(R)$ is a $p$-Boolean algebra \cite[Section 3.8]{schwartz_book}, and hence $\F_p(f)$ is a flat $T^0_E(R)$-module (since any module over a $p$-Boolean algebra is flat). Since $R$ is finitely generated as an algebra, there are only finitely many $\cK$-maps $f \colon R \to H_E^*$. By \cite[Theorem 3.8.6]{schwartz_book} there is an isomorphism $T_E^0(R) \cong \F_p^{\Hom_{\cK}(R,H_E^*)}$ (of $p$-Boolean algebras even), and hence  
\[ 
T_E(R) \cong T_E(R) \otimes_{T_E^0(R)}T_E^0(R) \cong \bigoplus_{f \in \Hom_{\cK}(R,H_E^*)}T_E(R) \otimes_{T_E^0(E)}\F_p(f).  
\] 
In other words, the $T$-functor splits into components:

\begin{equation}\label{eq:tfunctorcomponent}
T_E(R) \cong \bigoplus_{f \in \Hom_{\cK}(R,H_{E}^*)}T_E(R;f). 
\end{equation}
For $M \in R-\cU$ and $f \in \Hom_{\cK}(R,H_E^*)$ we similarly define 
\[
T_E(M;f) = T_EM \otimes_{T_E^0R}\F_p(f) \cong T_EM \otimes_{T_ER}T_E(R;f).
\]
Then $T_EM$ admits a similar decomposition:
\[
T_E(M) \cong \bigoplus_{f \in \Hom_{\cK}(R,H_{E}^*)}T_E(M;f). 
\]
Let $q_f \colon T_E(M) \to T_E(M;f)$ denote the projection map. We can then define a morphism $\rho_{M,(E,f)} \colon M \to T_E(M;f)$ as the composite $q_f \circ \rho_{M,E}$. The target $T_E(M;f)$ has the structure of an unstable $T_E(R;f)$-module, and the map $\rho_{R,(E,f)}$ makes $T_E(M;f)$ into an unstable $R$-module. In fact, for $M \in R_{fg}-\cU$ there is a functor $\bA_R \to R_{fg}-\cU$, given by assigning to $(E,f)$ the component $T_E(M;f)$, see \cite[Section 1.4]{Henn1996Commutative}. 

We also define $\eta_{M,(E,f)} \colon M \to H_E^* \otimes T_E(M;f)$ as the composite $q_f \circ \eta_{M,E}$. Finally, the map $\kappa_{M,E}$ descends to a morphism $\kappa_{M,(E,f)} \colon T_E(M;f) \to H_E^* \otimes T_E(M;f)$. This still defines a comodule structure on $T_E(M;f)$, see \cite[p.~46]{HennLannesSchwartz1995Localizations}.

The following is a consequence of \Cref{lem:comm}.
\begin{lem}\label{lem:commcomponent}
	There is an isomorphism
	\[
\eta_{M,(E,f)} \cong \kappa_{M,(E,f)} \circ \rho_{M,(E,f)}
	\]
  \begin{rem}\label{rem:spaces}
    For any space $X$, there is an evaluation map $BE \times \Map(BE,X) \to X$, which induces $H^*(X) \to H_E^* \otimes H^*\Map(BE,X)$. Taking adjoints, we obtain a map
    \[
T_EH^*(X) \to H^*\Map(BE,X). 
    \]
    In fact, for each map $f \colon BE \to X$, we obtain a map $T_E(H^*(X);f) \to H^*(\Map(BE,X)_{f})$, which makes the diagram 
    \[
\begin{tikzcd}
  H^*(X) \arrow[swap]{d}{\rho_{H^*(X),(E,f)}} \arrow[equal]{r} & H^*(X) \arrow{d} \\
  T_E(H^*(X);f) \arrow{r} & H^*(\Map(BE,X)_f),
\end{tikzcd}
    \]
    commute, where the right hand vertical morphism is induced by evaluation $\Map(BE,X)_f \to X$. Under some assumptions \cite[Corollary 3.4.3]{lannes_ihes} the bottom morphism is an equivalence. In particular, this holds if $X$ is a $p$-complete space such that $H^*(X)$ is of finite type, and that $T_E(H^*(X);f)$ is of finite type, and $T_E(H^*(X);f)$ is zero in degree 1. 
  \end{rem}
\end{lem}
\subsection{Central objects of a Noetherian unstable algebra}
Given a group $G$, Quillen introduced the category $\bA_G$ with objects $E < G$ elementary abelian subgroups of $G$, and morphisms the monomorphisms $E \to E'$ given by subconjugation in $G$. For a given unstable algebra, Rector \cite{Rector1984Noetherian} introduced the category $\bA_R$ we define now - as we shall see, when $R = H_G^*$ is the cohomology of a compact Lie group, then there is an equivalence of categories, $\bA_G \simeq \bA_{H_G^*}$.
\begin{defn}
Let $R$ be an unstable algebra over the Steenrod algebra. The category $\mathbf{A}_R$ is the category with objects pairs $(E,f)$ where $E$ is an elementary abelian $p$-group and $f \colon R \longrightarrow H_E^*$ is a morphism of unstable algebras such that $H_E^*$ is a finitely generated $R$-module via $f$. A morphism $(E_1,f_1) \longrightarrow (E_2,f_2)$ is a group homomorphism $\alpha \colon E_1 \longrightarrow E_2$ such that $f_1= H^*(\alpha)f_2$.
\end{defn}

We have the following two results about the category $\bA_R$. First we note that any Noetherian unstable algebra $R$ has finite Krull dimension $\dim(R)$. 
\begin{prop}\label{prop:rectorprops} Let $R$ be a Noetherian unstable algebra with Krull dimension $r$. 
	\begin{enumerate}
		\item The category $\bA_R$ has a finite skeleton. 
		\item There is an $\mathcal{F}$-isomorphism
		\[
R \longrightarrow \varprojlim_{\bA_R} H_E^*
		\]
		and hence $\dim(R) = \max\{ \dim(E) \mid (E,f) \in \bA_R \}$ and  $\dim(E) \le r$ for each $(E,f) \in \bA_R$. 
		\item If $(E_1,f_1) \longrightarrow (E_2,f_2)$ is a morphism in $\bA_R$, then $\alpha \colon E_1 \longrightarrow E_2$ is a monomorphism. 
	\end{enumerate}
\end{prop}
\begin{proof}
For (1) see \cite[Proposition 2.3]{Rector1984Noetherian}. The $\mathcal{F}$-isomorphism in (2) is due to Rector \cite[Theorem 1.4]{Rector1984Noetherian} for $p = 2$ and Broto--Zarati \cite[Theorem 1.3]{BrotoZarati1988Nillocalization} for $p$ odd. Quillen's argument \cite[Section 7]{Quillen1971spectrum} shows that
\[
\dim(R) = r=\max \{ \dim(E) \mid (E,f) \in \bA_R \}. 
\]
 This also implies the last claim in (2). Finally, (3) is a consequence of \cite[Corollary 2.4]{Quillen1971spectrum}. 
\end{proof}

We now come to the crucial definition of a central object, which is due to Dwyer and Wilkerson in the case $M = R$ \cite{DwyerWilkerson1992cohomology}. 
\begin{defn}\label{defn:centralobject}
	Let $R$ be an unstable algebra, then a non-zero pair $(E,f) \in \bA_R$ is central if $\rho_{R,(E,f)} \colon R \to T_E(R;f)$ is an isomorphism. More generally, if $M \in R-\cU$, then we say that $(E,f) \in \bA_R$ is $M$-central if $\rho_{M,(E,f)} \colon M \to T_E(M;f)$ is an isomorphism. 
\end{defn}
\begin{ex}\label{ex:central}
	For any map $f \colon H_E^* \to H_E^*$, the map $\rho_{H_E^*,(E,f)} \colon H_E^* \to T_E(H_E^*,f)$ is an isomorphism. This is a consequence of Lannes' computation of $T_E(H_E^*)$ \cite[Section 3.4.4]{lannes_ihes}, see \cite[Lemma 3.3]{DwyerMillerWilkerson1992Homotopical}. 
\end{ex}
\begin{ex}\label{ex:pgroup}
Let $G$ be a finite $p$-group, and $E < G$ an elementary abelian $p$-subgroup. We obtain an induced morphism $\res_{G,E}^* \colon H_G^* \to H_E^*$, and the pair $(E,\res_{G,E}^*) \in \bA_{H_G^*}$. Then, $H_G^* \to T_E(H_G^*;\res_{G,E}^*)$ is central if and only if $E<G$ is a central subgroup. The only if direction follows from Lannes' computation of $T_EH_G^*$ \cite{lannes_ihes}, while the converse is a theorem of Mislin \cite{Mislin1992Cohomologically}. For a general compact Lie group, the only if direction is still true, but the converse is false. We refer the reader to \Cref{sec:borel} for the details. 
\end{ex}
\begin{rem}
	Dwyer and Wilkerson use slightly different terminology; what we call central, they call a central monomorphism. 
\end{rem}
One of the key properties of Broto--Henn's proof of Duflot's depth theorem is that for any central subgroup $C$ of a compact Lie group, $H_G^*$ is a $H_C^*$-comodule. A similar result occurs for general unstable algebras. 
\begin{prop}\label{prop:comdule_central}
	Let $R$ be a connected Noetherian unstable algebra, and let $M \in R-\cU$. If $(C,g)$ is a non-trivial $M$-central object in $\bA_R$, then $M$ is a $H_C^*$-comodule.
\end{prop}
\begin{proof}
	We recall that $\kappa_{M,(C,g)} \colon T_C(M;g) \to H_C^* \otimes T_C(M;g)$ gives $T_C(M;g)$ the structure of a $H_C^*$-comodule. Since $(C,g)$ is $M$-central, $\rho_{M,(C,g)} \colon M \to T_C(M;g)$ is an isomorphism. It follows that the composite  \[
\xymatrix@C=1.2cm{\Psi_{M} \colon M \ar[r]_-{\cong}^-{\rho_{M,(C,g)}} & T_C(M;g) \ar[r]^-{\kappa_{M,(C,g)}}& H_C\otimes T_C(M;g)\ar[r]_-{\cong}^-{1 \otimes \rho^{-1}_{M,(C,g)}} &   H_C^* \otimes M
} 
 \]
 gives $M$ the structure of a $H_C^*$-comodule. 
\end{proof}
	As remarked previously, we do not need the full strength of the previous proposition, but rather just the following corollary, which we state for emphasis. 
\begin{cor}\label{cor:comodidentity}
	The composite 
	 \[
\xymatrix{M \ar[r]^-{\Psi_M} & H_C^* \otimes M \ar[r]^-{\epsilon_C \otimes 1} & M}
 \]
 is the identity. 
\end{cor}
\section{Depth and detection}\label{sec:depth}
In this section we prove versions of the theorems of Duflot and Carlson, using the theory of unstable algebras and the $T$-functor studied in the previous section. 
\subsection{Depth and regular sequences}\label{sec:depth_algebra}
We begin by recall the concepts of depth and regular sequences. We recall that we always assume that $R$ is a connected unstable algebra, whose underlying $\F_p$-algebra is finitely generated and graded-commutative, and we let $\frak m = R^{+}$ denote the maximal ideal generated by elements in positive degrees. Occasionally we will write $\frak m_R$ to make the dependence on $R$ clear. 

Let $M$ be an $R$-module, then an $M$-regular sequence is a sequence $y_1,\ldots,y_n$ in $\frak m$ such that $y_i$ is a non-zero divisor on $M/(y_1,\ldots,y_{i-1})$ for $i = 1,\ldots, n$. If $M$ is finitely generated over $R$, we say that the depth of $M$, $\depth_R(M)$, is the supremum of the length of all $M$-regular sequences in $\frak m$.\footnote{If $M$ is not finitely generated, there are alternative definitions of depth, however all modules we consider in this paper will be finitely generated.} One can show that the maximal length of any $M$-regular sequence is unique, and any $M$-regular sequence can be extended to one of this maximal length, see e.g., \cite[Remark 12.2.3]{Carlson1995Depth}, or \cite[Theorem A.24]{Poulsen}.

For a proof of the following, see \cite[Proposition 12.2.1]{CarlsonTownsleyValeriElizondoZhang2003Cohomology}. 
\begin{lem}\label{lem:depth_regular}
  Let $M$ be a finitely generated $R$-module. A sequence $y_1,\ldots,y_m \in \frak m$ of homogeneous elements of $\frak m$ is an $M$-regular sequence if and only if $y_1,\ldots,y_m$ are algebraically independent in $R$ and $M$ is a free module over the polynomial subring $k[y_1,\ldots,y_m] \subseteq R$. 
\end{lem}
We recall that the $\frak{m}$-torsion in an $R$-module $M$ is defined as 
\[
H_{\frak m}^0(M) = \{ x \in M \mid \text{ there exists } n \in \mathbb{N} \text{ such that } \frak{m}^nx = 0  \}
\]
The functor $H_{\frak m}^0(-)$ is left exact, and we define the local cohomology groups $H^i_{\frak m}(-)$ to be the right derived functors of $H_{\frak m}^0$. In fact, our module $M$ will usually be graded, so that our local cohomology groups of $M$ is bigraded, but we will usually suppress one grading. A reference for details on local cohomology is, for example, the book of Brodmann and Sharp \cite{BrodmannSharp2013Local}. We point out that $H^i_{\frak m}(M) \cong H^i_{\sqrt{\frak m}}(M)$ \cite[Proposition 7.3(2)]{24hours}. 

Local cohomology is related to depth is the following way, see \cite[Theorem 9.1]{24hours}. 
\begin{lem}\label{lem:depthlocalcohomology}
  If $M$ is a finitely generated $R$-module, then
  \[
\depth_R(M)  = \inf\{j \mid H_{\frak m}^j(M) \ne 0.\}
  \]
\end{lem}
Now suppose we have two connected Noetherian unstable algebras $R$ and $R'$, and a finite (graded) homomorphism $f \colon R \to R'$ (i.e., $R'$ is a finitely generated $R$-module via $f$). Because $f$ is finite, we have $\sqrt{\frak m_{R'}R} = \sqrt{\frak m_R}$. Using the independence theorem for local cohomology (\cite[Theorem 4.2.1]{BrodmannSharp2013Local} or \cite[Theorem 14.1.7]{BrodmannSharp2013Local} in the graded case) we then deduce the following from \Cref{lem:depthlocalcohomology}. 
\begin{lem}\label{lem:fgdepth}
  Let $R$ and $R'$ be connected Noetherian unstable algebras, and $f \colon R \to R'$ a finite homomorphism. Let $M$ be a finitely generated $R'$-module, then 
  \[
  \depth_R(M) = \depth_{R'}(M)
  \]
  where $M$ is an $R$-module by restriction of scalars. In particular, 
  \[
\depth_R(R') = \depth(R'). 
  \]
\end{lem}
A direct proof of this is the graded connected case is also given in \cite[Proposition 5.6.5]{NeuselSmith2002Invariant}.
\begin{rem}
  If the reader is worried about the use of graded-commutative, as opposed to strictly commutative graded rings, we note the following: A graded-commutative Noetherian ring $R$ is finitely generated over the strictly commutative subring $R^{\text{ev}}$, and the depth of $R$ is equal to the depth of $R$ as a module over the commutative graded ring $R^{\text{ev}}$. One way to see this is to use the characterization of depth in terms of local cohomology. Then, if we let $ \tilde{\frak m}$ denote the maximal ideal of $R^{\text{ev}}$, we have $\sqrt{ \tilde{\frak m}R} = \frak m$, so that $H_{ \tilde{\frak m}}^i(M) = H_{\frak m}^i(M)$. 
\end{rem}
\subsection{Duflot's depth theorem}
We are now in a position to prove our general version of Duflot's theorem. The proof follows closely the version given in \cite[Theorem 12.3.3]{CarlsonTownsleyValeriElizondoZhang2003Cohomology} or \cite[Theorem 3.17]{Totaro2014Group}. The reader should keep the example of $R = H_G^*$ and $S = H_G^*(X)$ for $G$ a compact Lie group and $X$ a finite $G$-CW complex in mind. 
\begin{thm}\sloppy\label{thm:duflot_regular}
  Let $R$ and $S$ be connected Noetherian unstable algebras, and $f \colon R \to S$ a finite $\cal{K}$-morphism, so that $S \in R_{fg}-\cal{U}$.  Suppose further that $(C,g) \in \bA_R$ is $S$-central. If $x_1,\ldots,x_n \in \frak m_R$ is a sequence of homogeneous elements in $R$ such that $x_1,\ldots,x_n$ form an $R$-regular sequence in $H_C^*$ (considered as an $R$-module via $gf$), then $x_1,\ldots,x_n$ is an $R$-regular sequence in $S$. 
  \end{thm}
\begin{proof}
  We recall that $S$ is a $H_C^*$-comodule, and in particular that the composite
	 \[
\xymatrix{S \ar[r]^-{\Psi_S} & H_C^* \otimes S \ar[r]^-{\epsilon_C \otimes 1} & S}
 \]
  is the identity (\Cref{cor:comodidentity}). The composite above is a sequence of $R$-modules, where we use $f \colon R \to S$ to induce the $R$-module structures. In particular, $S$ is a summand of $H_C^* \otimes S$ as an $R$-module, and hence also as an $\F_p[x_1,\ldots,x_n]$-module. We claim that $H_C^* \otimes S$ is a free $\F_p[x_1,\ldots,x_n]$-module. Note that this implies that $S$ is a projective $\F_p[x_1,\ldots,x_n]$-module, and hence a free $\F_p[x_1,\ldots,x_n]$-module \cite[Propositon 1.5.15]{BrunsHerzog1993CohenMacaulay}. By \Cref{lem:depth_regular} the sequence $x_1,\ldots,x_n$ is a regular sequence in $S$ as claimed. 

To see that $H_C^* \otimes S$ is a free $\F_p[x_1,\ldots,x_n]$-module, let $S^{\ge j} \subseteq S$ denote the submodule generated by homogeneous elements of degree at least $j$. This gives rise to a filtration of $H_C^* \otimes S$ by the $\F_p[x_1,\ldots,x_n]$-submodules $H_C^* \otimes S^{\ge j}$, with filtration quotients $H_C^* \otimes S^j$. 

  By \Cref{lem:depth_regular} and assumption $H_C^*$ is free over $\F_p[x_1,\ldots,x_n]$, and hence $H_C^* \otimes S^j$ is a finitely generated free $\F_p[x_1,\ldots,x_n]$-module (since in each degree $S^j$ is a finite-dimensional $\F_p$-vector space). This implies that the short exact sequences defining the filtration quotients all split, and so inductively we deduce that $H_C^* \otimes S$ is a free $\F_p[x_1,\ldots,x_n]$-module, as required.
\end{proof}
Duflot's theorem for an unstable algebra is an easy consequence. 
\begin{cor}\label{cor:duflot}
    Let $R$ and $S$ be connected Noetherian unstable algebras, and $f \colon R \to S$ a finite morphism, so that $S \in R_{fg}-\cal{U}$.  Suppose further that $(C,g) \in \bA_R$ is $S$-central, then $\depth_R(S) \ge \rank(C)$. 
 \end{cor}
\begin{proof}
  As is known, the cohomology of $H_C^*$ is Cohen--Macaulay. Moreover, by assumption $H_C^*$ is finitely generated as a $R$-module. It follows from \Cref{lem:fgdepth} that $\depth_{R}H_C^* = \depth H_C^* = \rank(C)=c$. This implies that there exist $x_1,\ldots,x_c$ in $R$ such that $x_1,\ldots,x_c$ form an $R$-regular sequence in $H_{C}^*$. By \Cref{thm:duflot_regular} the sequence $x_1,\ldots,x_c$ forms an $R$-regular sequence in $S$ via $f$, and so $\depth_R(S) \ge c$.
\end{proof}
\begin{ex}
	Let $R = H^*(\stc)$, the cohomology of the 3-connected cover of $S^3$. We recall that $\stc$ fits into a principal fibration $BS^1 \to \stc \to S^3$. There is a single non-trivial object in $\bA_R$, namely $(\Z/p,f)$, where $f \colon H^*(\stc) \to H^*_{\Z/p}$ is the map induced by the inclusion $\Z/p \subset S^1$. Moreover, we have $T_{\Z/p}(H^*(\stc);f) \cong H^*(\stc)$, see \cite[Section 3]{AguadeBrotoNotbohm1994Homotopy}. It follows that $\depth (H^*(\stc)) \ge 1$. In fact, $\dim( H^*(\stc)) \le 1$ (\Cref{prop:rectorprops}), so we must have that $\depth (H^*(\stc)) =1$. Of course, one calculates directly that $H^*(\stc) \cong \F_p[x] \otimes \Lambda_{\F_p}(y)$ is Cohen--Macaulay of dimension 1.
\end{ex}
The following is the algebraic incarnation of the fact that if $G$ is a group, then $E$ is central in $C_G(E)$. 
  \begin{prop}
    Let $R$ be a connected unstable Noetherian algebra and $(E,f) \in \bA_R$ then $\depth (T_E(R;f)) \ge \rank(E)$. 
  \end{prop}
\begin{proof}
  It suffices to show that there exists a map $h \colon T_E(R;f) \to H_E^*$ such that $(E,h)$ is central in $\bA_{T_E(R;f)}$. Such a map $h$ is constructed in the proof of Theorem 3.6 of \cite{DwyerMillerWilkerson1992Homotopical}. There is a commutative diagram \[
\begin{tikzcd}
T_E(R,f) \arrow{r}{T_E(f)}  & T_E(H_E^*,\text{id})  \\
R \arrow{u}{\rho_{R,f}} \arrow{r}[swap]{f} & H_E^* \arrow{u}[swap]{\rho_{H_E^*,\text{id}}}.
\end{tikzcd}
\]  
By \Cref{ex:central} $\rho_{H_E^*,\text{id}}$ is an isomorphism, and so we define $h = \rho_{H_E^*,\text{id}}^{-1}T_E(f)$. Since $f$ and $\rho_{R,f}$ are finite, it follows $h$ must be as well, and so $(E,h) \in \bA_{T_E(R,f)}$. 

To see that $(E,h)$ is central we will use the criterion given in \cite[Proposition 3.4]{DwyerWilkerson1992cohomology}, so that we must produce a morphism $T_E(R,f) \longrightarrow H_E^* \otimes T_E(R;f)$ such that the composition with the projections $T_E(R;f) \longrightarrow \F_p$ and $H_E^* \longrightarrow \F_p$ gives $h$ and the identity map of $T_E(R;f)$ respectively. We claim that $\kappa_{R,f}$ has this property. Indeed, because $\kappa_{R,f}$ makes $T_E(R;f)$ into a $H_E^*$-comodule, composition with $H_E^* \longrightarrow \F_p$ is the identity. For the other composite, consider the commutative diagram
  \[
  \begin{tikzcd}
{T_E(R,f)} \arrow[r, "{\kappa_{R,f}}"] \arrow[d, "T_E(f)"'] & {H_E^* \otimes T_E(R,f)} \arrow[r, "1 \otimes \epsilon"] \arrow[d, "1 \otimes T_E(f)"] & H_E^* \arrow[equal,d] \\
{T_E(H_E^*,\text{id})} \arrow[r, "{\kappa_{H_E^*,\text{id}}}"']      & {H_E^* \otimes T_E(H_E^*,\text{id})} \arrow[r, "1 \otimes \epsilon"']                        & H_E^*          
\end{tikzcd}
\]
By \Cref{lem:commcomponent,ex:central} we have $\kappa_{H_E^*,\text{id}} \cong \eta_{H_E^*,\text{id}}\circ\rho_{H_E^*,\text{id}}^{-1}$. Using the observation that $(1 \otimes \epsilon) \circ \eta_{H_E^*,\text{id}} \simeq \text{id}$, and commutativity we then see that
\[
\begin{split}
  (1 \otimes \epsilon) \circ \kappa_{R,f} &\cong (1 \otimes \epsilon) \circ \kappa_{H_E^*,\text{id}} \circ T_E(f) \\
  &  \cong (1 \otimes \epsilon) \circ \eta_{H_E^*,\text{id}} \circ \rho_{H_E^*,\text{id}}^{-1} \circ T_E(f) \\
  & \cong \text{id} \circ \rho_{H_E^*}^{-1} \circ T_E(f)\\
  & = h,
\end{split}
\]
as required. 
\end{proof}

	One can give a variant of \Cref{cor:duflot}, using similar assumption to those of \cite[Theorem 1.2]{DwyerWilkerson1992cohomology}. In this case, the reader should think of $R = H_G^*$ and $P = H_S^*$ for $G$ a finite group, and $S$ a Sylow $p$-subgroup.  
\begin{thm}\label{thm:duflotretract}
		Suppose there exists a map $i \colon R \to P$ of unstable algebras such that:
	\begin{enumerate}
		\item Both $R$ and $P$ are finitely generated as algebras, and the map $i$ makes $P$ into a finitely generated module over $R$.
		\item The map $i$ has an additive left inverse $P \to R$ which is a map of $R$-modules.
		\item There exists a non-trivial central object $(C,g) \in \bA_P$. 
	\end{enumerate}
	Then the following hold:
	\begin{enumerate}
	 	\item If $x_1,\ldots,x_n \in \frak m_R$ is a sequence of homogeneous elements such that the restriction of $x_1,\ldots,x_n$ form an $R$-regular sequence in $H_C^*$, then $x_1,\ldots,x_n$ is an $R$-regular sequence in $R$. 
	 	\item The depth of $R$ is greater than or equal to the rank of $C$. 
	 \end{enumerate} 
\end{thm}
\begin{proof}
	Observe that $R$ is a direct summand of $P$ as an $R$-module, and hence if $P$ is a free module over $\F_p[x_1,\ldots,x_n]$, then so is $R$, and hence the sequence $x_1,\ldots,x_n$ is regular in $R$. The result then easily follows from \Cref{thm:duflot_regular,cor:duflot}. 
\end{proof}
\begin{ex}
  We recall from \Cref{rem:spaces} that for a space $X$ and a map $f \colon BE \to X$ there is a natural map $T_E(H^*(X);f) \to H^*(\Map(BE,X)_f$, which under some mild conditions on $X$ is an isomorphism. Say that $X$ is a Lannes' space if this is the case. We can define a category $\bA_X$ whose objects are pairs $(E,f)$ where $E$ is an elementary abelian $p$-group, and $f \colon BE \to X$ makes $H_E^*$ into a finitely generated $H^*(X)$-module, and a morphism $(E,f) \to (E',f')$ is a monomorphism $\alpha \colon E \to E'$ such that $f' \circ B\alpha \simeq f$. There is an obvious functor $\bA_X \to \bA_{H^*(X)}$. If $H^*(X)$ is of finite type, and $X$ is $p$-complete, then this is an equivalence by \cite[Theorem 3.1.1]{lannes_ihes}. Now suppose $X$ is additionally a Lannes' space. Then $(E,f^*) \in \bA_{H^*(X)}$ is central if and only if $H^*(X) \to H^*(\Map(BE,X)_f)$ is an equivalence. In this case, the mapping space $\Map(BE,X)_f$ plays the role of the centralizer of $f \colon BE \to X$ inside of $X$.
\end{ex}
\begin{rem}
	Suppose the conditions of \Cref{thm:duflotretract} are satisfied, and that moreover the map $i \colon S \to R$ is also a map of unstable modules over the Steenrod algebra, then Notbohm \cite{notbohm_depth} has shown that the depth of $R$ satisfies
	\[
\depth(R) = \min\{ \depth T_E(R;f) \mid (E,f) \in \bA^*_{R}\}.
	\]
Here $\bA^*_R$ is the full subcategory of $\bA_R$ consisting of $(E,f)$ where $E$ is a \emph{non-trivial} elementary abelian $p$-group. This result is not explicitly stated in \cite{notbohm_depth}, so we briefly explain how to deduce it. Define a functor $\alpha_R \colon \bA^*_R \to \cal{K}$ which sends $(E,f) \in \bA_R^*$ to $T_E(R;f)$. Then, by \cite[Theorem 1.2]{DwyerWilkerson1992cohomology} the assumptions imply that
		\[
\varprojlim{}^i {\alpha_R} \cong \begin{cases}
	R & i = 0 \\
	0 & i > 0. 
\end{cases}
		\]
		In particular, by \cite[Corollary 2.3]{notbohm_depth} we have $\depth R \ge \min\{\depth T_E(R;f) \mid (E,f) \in \bA^*_R \}$. On the other hand, by  \cite[Theorem 3.1]{notbohm_depth} we have $\depth R \le \depth T_E(R;f)$, so the previous inequality is actually an equality. 
\end{rem}
\subsection{Carlson's detection theorem}
In \cite{Carlson1995Depth} Carlson showed that if $G$ is a finite group, and $H_G^*$ has depth $s$, then the collection $\cal{H}_S= \{ C_G(E) \mid \rank(E) = s \}$ detects $H_G^*$, in the sense that the product of the restriction maps $H_G^* \to  \prod_{H \in \cal{H}_S} H_{H}^*$ is injective. In this section, we generalize this to the case of a connected Noetherian algebra $R$, and more generally for $M \in R_{fg}-\cU$. We will need the notion of the $T$-support of a module $M$ in $R-\cU$ and the $R-\cU$ transcendence degree of $M$, due to Henn \cite{Henn1996Commutative} and Powell \cite{powell}, respectively. 

\begin{defn}
  Let $R$ be a Noetherian unstable algebra, and $M \in R-\cal{U}$, then the $T$-support of $M$ is 
  \[
T-\supp(M) = \{ (E,f) \in \bA_R \mid T_E(M;f) \ne 0 \}. 
 \]
 The $R-\cal{U}$ transcendence degree of $M$ is 
 \[
\TrD_{R-\cU}(M) = \sup\{\rank( E) \mid (E,f) \in T-\supp(M)\}. 
 \]
\end{defn}
We have the following simple lemma, which follows from exactness of the $T$-functor.
\begin{lem}\label{lem:trdinjective}
	If $M \to N$ is a monomorphism in $R-\cU$, then $\TrD_{R-\cU}(M) \le \TrD_{R-\cU}(N)$. 
\end{lem}

The following is the key computation for Carlson's theorem. Here we let $\bA_R^t \subset \bA_R$ denote the full subcategory consisting of those $(E,f)$ which $\rank(E)  =t $. 

\begin{prop}\label{prop:supportcalc}
Let $R$ be a Noetherian unstable algebra, $M \in R_{fg}-\cU$, $n$ a natural number, and 
\[
N = \prod_{\stackrel{(E,f) \in \bA^t_R}{t < s}} H_E^* \otimes T_E(M;f)^{<n}.
\]
Then, $\TrD_{R-\cal{U}}(N) < s.$
 \end{prop}
 \begin{proof}
 Using exactness of the $T$-functor, it suffices to show that each component of the product has transcendence degree less than $s$. To see this, fix some $(W,j) \in \bA_R$, then by \cite[Lemma 3.6]{Henn1996Commutative} we have
    \begin{equation}\label{eq:henntfunctor}
T_W(H_E^* \otimes T_E(R;f)^{<n};j) \cong \prod_{\Hom_{\bA_R}((W,j),(E,f))} H^*_E \otimes T_E(R;f)^{<n}
    \end{equation}
    Recall that a morphism $(W,j) \to (E,f)$ in $\bA_R$ is in particular a monomorphism $W \to E$ of elementary abelian $p$-groups; in particular, if the rank of $W$ is greater than or equal to $s$, there are no such morphisms. We deduce that $\TrD_{R-\cU}(H_E^* \otimes T_E(R;f)^{<n}) < s$, as required. 
 \end{proof}
The proof of Carlson's theorem will use two results from the theory of unstable algebras over the Steenrod algebra. The first is a deep result of Henn, Lannes, and Schwartz \cite[Theorem 4.9]{HennLannesSchwartz1995Localizations} or \cite[Theorem 1.10]{Henn1996Commutative} for the exact form we require.
\begin{thm}[Henn--Lannes--Schwartz]\label{thm:hls}
	Let $R$ be a Noetherian unstable algebra, and $M \in R_{fg}-\cU$. Then there exists a natural number $n$ such that the maps $\eta_{M,(E,f)}$ induce a monomorphism
	\[
\xymatrix{M \ar[r]& \displaystyle \prod_{(E,f) \in \bA_R}H_E^* \otimes T_E(M;f)^{<n}}
	\]
	in the category $R_{fg}-\cU$. 
\end{thm}
The second result we require is due to Powell \cite[Proposition 7.3.1]{powell}. 
 \begin{prop}[Powell]\label{prop:powell_mono}
   Let $1 \ne M \hookrightarrow N$ be a monomorphism in $R_{fg}-\cU$, then 
   \[
\TrD_{R-\cU}(M) \ge \depth_R(N). 
   \]
 \end{prop}
Since Powell's work in not published, we give a proof of this. Our proof is slightly different from that given by Powell, but very similar in spirit. We recall the existence of Brown--Gitler modules $J_R(n)$ in the category $R-\cU$ (see \cite[Section 1.5]{Henn1996Commutative}), which represent the functor $M \mapsto (M^n)^*$, where $()^*$ denotes the dual. Given $(E,f) \in \bA_R$, we can then define an injective object $I_{(E,f)}(n)$ as $H_E^* \otimes J_{T_E(R;f)}(n)$ in $R-\cU$ \cite[Proposition 1.6]{Henn1996Commutative}. In fact,if $R$ is a Noetherian unstable algebra, then $I_{(E,f)}(n)$ is injective in $R_{fg}-\cU$. It is not hard to verify (for example, \cite[Lemma 6.1.7]{powell}) that we have
\begin{equation}\label{eq:tfunctorinjec}
\Hom_{R-\cU}(M,I_{(E,f)}(n)) \cong (T_E(M;f)^n)^*. 
\end{equation}
With this in mind, we give a proof of \Cref{prop:powell_mono}. 
 \begin{proof}
   By \Cref{thm:hls} we can find an embedding of $N$ (and hence of $M$) into a product of modules of the form $H_E^* \otimes T_E(M;f)^{<n}$. We assume that the embedding of $N$ is reduced, in the sense that no summand can be removed while still remaining injective. In this case, by the proof of Proposition 2 of \cite{Henn1998variant} we can assume that in the product each elementary abelian $p$-group $E$ has rank at least $\depth_R(N)$. Each of the modules $T_E(M;f)^{<n}$ can be embedded into a finite direct product of modules of form $J_{T_E(R;f)}(k)$ (for $k<n$, see the second proof of Theorem 1.9 of \cite{Henn1996Commutative}), and hence $M$ can be embedded into a product of injective modules of the form $I_{(E,f)}(k)$ with $\rank(E) \ge \depth_R(N)$. In particular, by \eqref{eq:tfunctorinjec}, we have $(E,f) \in T-\supp(N)$, and hence $\TrD_{R-\cU}(N) \ge \rank(E) \ge \depth_R(N)$. 
 \end{proof}

 We now prove our version of Carlson's theorem. 
 \begin{thm}\label{thm:carlson}
   Let $R$ be a Noetherian unstable algebra, and $M \in R_{fg}-\cU$. Suppose that $\depth_R(M) \ge s$, then the map
   \[
\xymatrix{
  \phi_s \colon M \ar[r] & \displaystyle \prod_{(E,f) \in \bA_R^s} T_E(M;f)
}
   \]
   given by the product of the maps $\rho_{M,(E,f)}$ is injective. 
 \end{thm}
 \begin{proof}
For any $t \ge 0$, we let $I_t$ denote the kernel of the map 
\[
\xymatrix{
  \phi_t \colon M \ar[r] & \displaystyle \prod_{(E,f) \in \bA_R^t} T_E(M;f),
}
\] so that our claim is that if $\depth_R(M) \ge s$, then $I_s$ is trivial. We will prove the contrapositive, namely that if $I_s$ is non-trivial, then $\depth_R(M) < s$. We first claim that if $I_s$ is non-trivial, then so is $I_t$ for $t\ge s$, and in particular that $I_s \subseteq I_t$ for each $t \ge s$. Indeed, let $(E',f') \in \bA_R^t$, then for $E < E'$ of rank $s$, the pair $(E,f)$ where $f=\res_{E',E}^*f'$ is in $\bA_R^{s}$. Moreover, there is a morphism $(E,f) \to (E,f')$ in $\bA_R$, which induces a commutative diagram
\[
\begin{tikzcd}
M \arrow[d, "\rho_{M,(E,f)}"'] \arrow[dr, "\rho_{M,(E',f')}"] & \\
T_E(M;f) \arrow{r}                                 &                 T_{E'}(M;f').
\end{tikzcd}
\]
Here the commutativity comes from the fact that the inclusion $\{ e \} \hookrightarrow E'$ factors as $\{ e \} \hookrightarrow E \hookrightarrow E'$. In particular, we have $\ker(\rho_{M,(E,f)}) \subseteq \ker(\rho_{M,(E',f')})$. It easily follows that $I_s \subseteq I_t$ for each $t \ge s$. 

Using \Cref{thm:hls} we can choose $n$ large enough so that the morphism 
\[
\xymatrix{\lambda \colon M \ar[r]& \displaystyle \prod_{(E,f) \in \bA_R}H_E^* \otimes T_E(M;f)^{<n}}
\]
is a monomorphism in $\cal{R}_{fg}-\cU$. Here the map is induced by the product of the maps $\eta_{R,(E,f)}$. Recall from \Cref{lem:commcomponent} that $\eta_{M,(E,f)} \cong \kappa_{M,(E,f)} \circ \rho_{M,(E,f)}$, so that $\lambda$ factors through the product of the maps $\rho_{M,(E,f)} \colon M \to T_E(M;f)$.  

We factor $\lambda$ as a product $\lambda = \lambda_{\ge s} \times \lambda_{< s}$ where
\[
\xymatrix{\lambda_{\ge s} \colon M \ar[r]& \displaystyle \prod_{\stackrel{(E,f) \in \bA^t_R}{t \ge s}}H_E^* \otimes T_E(M;f)^{<n}}
\]
and
\[
\xymatrix{\lambda_{< s} \colon M \ar[r]& \displaystyle \prod_{\stackrel{(E,f) \in \bA^t_R}{t < s}}H_E^* \otimes T_E(M;f)^{<n}}
\]
Note that $\lambda_{\ge s}$ is the product of the maps $\phi_t$ for $t \ge s$. In particular, by the discussion above, we see that $I_s$ is contained in the kernel of $\lambda_{\ge s}$.   Furthermore, since $\lambda$ is injective, we deduce that the restriction of $\lambda_{< s}$ to $I_s \subseteq M$ is injective. By \Cref{lem:trdinjective} we have $\TrD_{R-\cU}(I_s) \le \TrD_{R-\cU}(N)$, where
\[
N = \displaystyle \prod_{\stackrel{(E,f) \in \bA^t_R}{t < s}}H_E^* \otimes T_E(R;f)^{<n}. 
\]
By \Cref{prop:supportcalc} we deduce that $\TrD_{R-\cU}(I_s)<s$, and hence by \Cref{prop:powell_mono} we have $\depth_R(M) \le \TrD_{R-\cU}(I_s)<s$, as required. 
 \end{proof}

\section{Examples}\label{sec:examples}
We finish with examples from group theory, homotopical group theory, and modular invariant theory. 
\subsection{Borel equivariant cohomology}\label{sec:borel}
The original theorems of Duflot and Carlson were proved for the cases of finite groups, and then later extended to the case of compact Lie groups. Here we extend these to some further group theoretic situations. We note that the proof of Duflot's theorem given by Broto and Henn \cite{BrotoHenn1993Some} can be used for all these cases below, however Carlson's proof relies on properties of a suitable transfer, which does not seem to exist in general. 

We begin with the relevant $T$-functor computations and an identification of Rector's category $\bA_{H_G^*}$. 
\begin{thm}\label{thm:tcalcsgroups}
	Assume we are in one of the following cases:
	\begin{enumerate}[label=(\alph*)]
		\item $G$ is a compact Lie group, and $X$ is a $G$-$CW$-complex with finitely many $G$-cells. 
		\item $G$ is a discrete group for which there exists a mod $p$ acyclic $G$-CW complex with finitely many $G$-cells and finite isotropy groups, and $X$ is any $G$-$CW$ complex with finitely many $G$-cells and with finite isotropy groups. 
	\end{enumerate}
	Then the following hold. 
	\begin{enumerate}
		\item The cohomology $H_G^*$ is an unstable Noetherian algebra, and there is an equivalence of categories $\bA_G \simeq \bA_{H_G^*}$ given by associating to $E \le G$ the pair $(E,\res_{G,E}^*)$ where $\res_{G,E}^*$ is the restriction homomorphism $H_G^* \to H_E^*$. 
		\item The cohomology $H_G^*(X)$ is an unstable finitely generated $H_G^*$-module, and there is an  isomorphism
		\[
T_E(H_G^*(X);\res_{G,E}^*) \cong H_{C_G(E)}^*(X^E). 
		\]
    \end{enumerate}
\end{thm}
\begin{proof}
In case (1), the finite generation is due to Quillen \cite{Quillen1971spectrum}, while the equivalence of the categories was shown by Rector \cite[Proposition 2.6]{Rector1984Noetherian}, see also \cite[Section I.5.3]{HennLannesSchwartz1995Localizations}. The $T$-functor computation is due to Lannes, in unpublished work \cite{lannes_unpublished}, however see also the notes of Henn \cite{henn_notes}. (2) is shown by Henn \cite[Appendix A]{Henn1996Commutative}. 
\end{proof}
\sloppy
We recall that there is a morphism $\rho_{H_G(X)^*,(E,\res_{G,E}^*)} \colon H_G^*(X) \to T_E(H_G^*(X);\res_{G,E}^*) \cong H_{C_G(E)}^*(X^E)$. This agrees with the usual restriction map $\res^*_{(G,X),(C_G(E),X^E)} \colon H_G^*(X) \to H_{C_G(E)}^*(X^E)$. We say that $E \le G$ is $X$-cohomologically $p$-central if it is $H_G^*(X)$-central in the sense of \Cref{defn:centralobject}, i.e., if $\res_{(G,X),(C_G(E),X^E)}^*$ is an isomorphism. In the special case where $X$ is a point, we simply call this cohomologically $p$-central.  From the previous theorem we deduce the following. 
\begin{cor}
	If $E < G$ is a central elementary abelian $p$-subgroup of $G$ that acts trivially on $X$, then $E$ is $X$-cohomologically $p$-central.
\end{cor}
\begin{rem}
	If $G$ is a finite $p$-group, then any cohomologically $p$-central subgroup $E \le G$ is a central elementary abelian $p$-subgroup, but this false in general for compact Lie groups. An example is given at $p=2$ by the inclusion of a Sylow 2-subgroup in $\Sigma_3$ (which has trivial center). The point is that $(B\Sigma_3)^\wedge_2 \simeq B\Z/2$, so $\F_2$-cohomology cannot tell the difference between the two. In general, the maximal cohomolgically $p$-central subgroup of a compact Lie group $G$ is the maximal central elementary abelian $p$-subgroup of $G/\cal{O}_{p'}(G)$, where $\cal{O}_{p'}(G)$ is the largest normal $p'$-subgroup of $G$, see \cite[Theorem 1]{Mislin1992Cohomologically}. 
\end{rem}
We now wish to apply \Cref{thm:duflot_regular} with $R = H_G^*$ and $S = H_G^*(X)$, with $f \colon H_G^* \to H_G^*(X)$ the map induced by sending $X$ to a point. From the discussions above, we deduce the following. 
\begin{thm}
	Let $G$ and $X$ be as in \Cref{thm:tcalcsgroups}. Let $C < G$ be any non-trivial central elementary abelian $p$-subgroup that acts trivially on $X$, and write $\frak m$ for the maximal ideal of $H_G^*$ generated by homogeneous elements of positive degree. 
	\begin{enumerate}
	 	\item If $x_1,\ldots,x_n \in \frak m$ is a sequence such that the restrictions of $x_1,\ldots,x_n$ form a $H_G^*$-regular sequence in $H_C^*$, then $x_1,\ldots,x_n$ form a $H_G^*-$regular sequence in $H_G^*(X)$.
	 	\item The depth of $H_G^*(X)$ is at least $\rank(C)$. 
	 \end{enumerate} 
\end{thm}
In the case a finite group, (2) is precisely the original theorem proved by Duflot \cite{Duflot1981Depth}, while for compact Lie groups this is due to Broto and Henn \cite{BrotoHenn1993Some}. 

We also note that if $G$ is a compact Lie group and $X$ a point, then we can apply \Cref{thm:duflotretract}. Indeed, let $T$ be a maximal torus in $G$, $N(T)$ the normalizer of $T$ in $G$, $N_p(T)$ the inverse image in $N(T)$ of a Sylow $p$-subgroup of $N(T)/T$. Then we can take the map $i \colon H_G^* \to H_{N_p(T)}^*$ to be the natural restriction map. By the Becker--Gottlieb transfer, this map has a retract that is a map of $H_G^*$-modules, and a map of unstable modules, and $N_p(T)$ has a non-trivial center, see \cite[Proposition 1.3]{DwyerWilkerson1992cohomology}. We thus deduce the following. 
\begin{thm}
Let $G$ be a compact Lie group, $C_{N_p(T)}$ the maximal central elementary abelian $p$-subgroup of $N_p(T)$, and $C \le C_{N_p(T)}$ a central elementary abelian $p$-subgroup of $N_p(T)$.  
	\begin{enumerate}
	 	\item If $x_1,\ldots,x_n \in \frak m$ is a sequence of homogeneous elements such that the restriction of $x_1,\ldots,x_n$ form a regular sequence in $H_C^*$, then $x_1,\ldots,x_n$ is a regular sequence in $H_{G}^*$. 
	 	\item The depth of $H_G^*$ is greater than or equal to the rank of $C_{N_p(T)}$. 
	 \end{enumerate} 	
\end{thm}
Finally, Carlson's depth theorem can be extended to all the classes considered in this section. This follows from \Cref{thm:carlson,thm:tcalcsgroups}.
\begin{thm}\label{thm:carlson_borel}
	Let $G$ and $X$ be as in \Cref{thm:tcalcsgroups}. If $\depth_{H_G^*}(H_G^*(X)) \ge s$, then the morphism
	\[
\xymatrix{H_G^*(X) \ar[r]& \displaystyle\prod_{\stackrel{E < G}{\dim E = s}}H_{C_G(E)}^*(X^E)}
	\]
	is injective.
\end{thm}
When $G$ is a finite group and $X$ is a point, we recover Carlson's original theorem. 
\subsection{Further group theoretic cases}
In this section, we give two further examples from group theory. In these cases, we do not have a computation for $H_G^*(X)$, but only for $H_G^*$ itself. In order to keep our notation compact, $H_G^*$ should be understood to be the continuous mod-$p$ cohomology when $G$ is a profinite group, e.g., if $G$ is the inverse limit of finite groups $G_i$, then $H_G^* \cong \colim H^*_{G_i}$. The third class we consider may not be familiar to the reader, so we repeat the definition due to Broto and Kitchloo \cite{BrotoKitchloo2002Classifying}. 
\begin{defn}[Broto--Kitchloo]
  Let $\cX$ be a class of compactly generated Hausdorff topological groups and let $p$ be a fixed prime. We define a new class of groups, denoted $\cal{K}_1\cal{X}$ as the class of compactly generated Hausdorff topological groups $G$ for which there exists a finite $G$-CW complex with the following properties:
  \begin{enumerate}
    \item The isotropy subgroups of $X$ belong to the class $\cX$. 
    \item For every finite $p$-subgroup $\pi < G$, the fixed point space $X^{\pi}$ is $p$-acyclic. 
  \end{enumerate}
\end{defn}
In particular, Kac--Moody groups belong to $\cal{K}_1\cX$ when $\cX$ is the class of compact Lie groups. 

We now state the relevant $T$-functor calculations. 
\begin{thm}\label{thm:group_point}
	Assume we are in one of the following cases:
	\begin{enumerate}[label=(\alph*)]
		\item $G$ is a profinite group such that the continuous mod $p$ cohomology $H_G^*$ is finitely generated as an $\F_p$-algebra. 
    \item $G$ is a group of finite virtual cohomological dimension. 
		\item $G$ is in $\cal{K}_1\cX$ where $\cX$ is the class of compact Lie groups (for example, a Kac--Moody group).
	\end{enumerate}
	Then the following hold. 
	\begin{enumerate}
		\item The cohomology $H_G^*$ is an unstable Noetherian algebra, and there is an equivalence of categories $\bA_G \simeq \bA_{H_G^*}$ given by associating to $E \le G$ the pair $(E,\res_{G,E}^*)$ where $\res_{G,E}^*$ is the restriction homomorphism $H_G^* \to H_E^*$. 
		\item
		There are isomorphisms
		\[
	T_E(H_G^*;\res_{G,E}^*) \cong H_{C_G(E)}^*. 	
		\]
	\end{enumerate}
\end{thm}
\begin{proof}
		The case of profinite groups is shown in \cite{Henn1998Centralizers}, see also \cite[Theorem 0.2(c)]{Henn1996Commutative}. Case (b) is due to Lannes \cite{lannes_unpublished}, see also \cite[Theorem 5.2]{HennLannesSchwartz1995Localizations} and the discussion on page 49 of \cite{HennLannesSchwartz1995Localizations}. 

				For (c), everything except the equivalence of categories is shown by Broto and Kitchloo \cite{BrotoKitchloo2002Classifying}, see Theorems A and B. The equivalence of the categories is shown in the same way as the compact Lie group case \cite[Section I.5.3]{HennLannesSchwartz1995Localizations}. The key point is that for any elementary abelian $p$-group $E$ there is an isomorphism \cite[Equation (8)]{BrotoKitchloo2002Classifying}
				\[
\Hom_{\cal{K}}(H_G^*,H_E^*) \cong \Rep(E,G)
				\]
				where $\Rep(E,G)$ is the quotient of $\Hom(E,G)$ by the conjugation action of $G$. 
\end{proof}
\begin{rem}
	The assumption that the profinite group $G$ is finitely generated holds in many interesting cases, for example profinite $p$-adic analytic groups. 
\end{rem}
With this in mind, the following theorems follow precisely as in the previous section. 
\begin{thm}
  Let $G$ be as in \Cref{thm:group_point}. Let $C < G$ be a central elementary abelian $p$-subgroup, and write $\frak m$ for the maximal ideal of $H_G^*$ generated by homogeneous elements of positive degree. 
  \begin{enumerate}
    \item If $x_1,\ldots,x_n \in \frak m$ is a sequence such that the restrictions of $x_1,\ldots,x_n$ form a regular sequence in $H_C^*$, then $x_1,\ldots,x_n$ form a $H_G^*-$regular sequence in $H_G^*$.
    \item The depth of $H_G^*$ is at least $\rank(C)$. 
   \end{enumerate} 
\end{thm}
Carlson's theorem takes the expected form. 
\begin{thm}\label{thm:carlson_group2}
  Let $G$ be as in \Cref{thm:group_point}. If $\depth(H_G^*) \ge s$, then the product of restriction maps 
         \[
\xymatrix{
H_G^* \ar[r] & \displaystyle \prod_{\stackrel{E < S}{\rank(E) = s}} H_{C_G{(E)}}^*
}
   \]
   is injective.
\end{thm}
        
\subsection{Fusion systems and \texorpdfstring{$p$}{p}-local group theory}\label{sec:plocal}
For a finite group $G$ with Sylow $p$-subgroup $S$, one can associate a category $\cal{F}_s(G)$ to be the category whose objects are the subgroups of $S$, and which has morphisms 
\[
\Hom_{\cal{F}_s(G)}(P,Q) = \Hom_G(P,Q)
\]
where the latter denotes the set of injective group homomorphisms which are induced by conjugation in $G$. This is known as the fusion category of $G$ over $S$. Many group theoretical theorems can be stated in terms of $\cal{F}_S(G)$, for example see \cite[Part I.1]{AschbacherKessarOliver2011Fusion}. 

More generally, we can define abstract fusion systems that do not come from finite groups. In fact, we want to capture the homotopy theory of compact Lie groups, and not just finite groups. To do so, we begin with a discrete $p$-toral group, that is a group $S$ which contains a normal subgroup $S_0$ such that $S/S_0$ is a finite $p$-group, and $S_0 \cong (\Z/p^{\infty})^r$ for some $r$. The following is \cite[Definition 2.1]{BrotoLeviOliver2007Discrete}, and is essentially due to Puig \cite{Puig2006Frobenius} in the case where $S$ is a finite $p$-group. 
\begin{defn}
	A fusion system $\cF$ on a discrete $p$-toral group is a category with objects subgroups of $S$, and whose morphisms $\Hom_{\cF}(P,Q)$ satisfy the following:
	\begin{enumerate}[(a)]
		\item $\Hom_S(P,Q) \subseteq \Hom_{\cal{F}}(P,Q) \subseteq \Inj(P,Q)$.
		\item Every morphism in $\cal{F}$ factors as an isomorphism in $\cal{F}$ followed by an inclusion. 
	\end{enumerate}
\end{defn}
In order to define a working theory, one needs to impose additional technical axioms so that the fusion system is a \emph{saturated} fusion system - we will not spell out precisely what it means for a fusion system to be saturated, but rather refer the reader to \cite[Definition 2.2]{BrotoLeviOliver2007Discrete}.

Given a pair consisting of a discrete $p$-toral group $S$, and a saturated fusion system $\cal{F}$ defined on $S$, Broto, Levi, and Oliver constructed a category $\cal{L}$, the centric linking system associated to $(S,\cF)$. The triple $\cal{G} = (S,\cF,\cL)$ is known as a $p$-compact group. To this, we can define a classifying space $B\cG$ to be the $p$-completed nerve $|\cal{L}|^{\wedge}_p$. It was shown by Chermak \cite{Chermak2013Fusion} when $S$ is a finite $p$-group, and in general by Levi and Libman \cite{LeviLibman2015Existence}, that the pair $(S,
\cF)$ uniquely determines the centric linking system $\cL$. Thus, we will usually refer to the $p$-local compact group as simply a pair $\cG = (S,\cF)$ consisting of a discrete $p$-toral group, and a saturated fusion system on $S$. 
\begin{ex}
	We finish this very brief introduction to the theory of $p$-local compact groups, by giving two examples. 
\begin{enumerate}[(a)]
	\item	Given a compact Lie group $G$, one can define a $p$-local compact group $\cal{G} = (S,\cal{F}_S(G))$ where $S \in \Syl_p(G)$ is a maximal discrete $p$-toral subgroup, and $\cal{F}_S(G)$ is the fusion system defined above for $S$ a finite $p$-group. Then, there is an equivalence $B\cal{G} \simeq BG^{\wedge}_p$. 
	\item Let $(X,BX,e)$ be a $p$-compact group as defined by Dwyer and Wilkerson \cite{DwyerWilkerson1994Homotopy}. That is, $X$ is an $\F_p$-finite space, $BX$ a pointed $p$-complete space, and $e\colon X \to \Omega BX$ a homotopy equivalence. Any $p$-compact group has a Sylow $p$-subgroup $S \xr{f} X$ which is a discrete $p$-toral group. We then define a fusion system $\cal{F}_{S,f}(X)$ whose objects are subgroups of $S$ and for which
	\[
\Hom_{\cal{F}_{S,f}(X)}(P,Q) = \{ \phi \in \Hom(P,Q) \mid Bf|_{BQ}\circ B\phi \cong Bf \mid_{BP} \}.
	\]
	The pair $(S,\cF_{S,f}(X))$ then defines a $p$-local compact group with classifying space homotopy equivalent to $BX$. 
\end{enumerate}
\end{ex}
We have the following structural results about the cohomology $H_{\cal{G}}^*$, which is defined to be the cohomology of the space $B\cal{G}$. In order to state this, we point out that there is a canonical map $\theta \colon BS \to B\cal{G}$, which is the analog of the inclusion of a $p$-Sylow subgroup. 
\begin{prop}\label{thm:cohomlogyplocal}
Let $\cal{G} = (S,\cal{F})$ be a $p$-local compact group. 
\begin{enumerate}
	\item Both $H_{\cG}^*$ and $H_S^*$ are finitely generated algebras, and the map $\theta^* \colon H_{\cG}^* \to H_S^*$ makes $H_S^*$ into a finitely generated $H_{\cG}^*$-module. 
	\item The map $\theta^*$ has an additive left inverse $t \colon H_{\cG}^* \to H_S^*$ which is a map of $H_{\cG}^*$-modules.
	\item The group $S$ has a non-trivial central element of order $p$. 
\end{enumerate}
\end{prop}
\begin{proof}
	That $H_S^*$ is a finitely generated $\F_p$-algebra is \cite[Theorem 12.1]{DwyerWilkerson1994Homotopy}, while part (2) is a consequence of \cite[Proposition 4.24]{bchv}.  The splitting then implies that $H_{\cG}^*$ is a finitely generated $\F_p$-algebra \cite[Corollary 4.26]{bchv}. Moreover, it is shown in \cite[Proposition 5.5]{bchv} that $H_S^*$ is a finitely generated $H_{\cG}^*$-module. Finally, that $S$ has a non-trivial central element of order $p$ is the same argument as in \cite[Proposition 1.3(c)]{DwyerWilkerson1992cohomology}; the conjugation action of $S$ on the elements of order $p$ in $(\Z/p^{\infty})^r$ must point wise fix a non-trivial subgroup. 
\end{proof}

We will also need the following results about discrete $p$-toral groups. 
\begin{prop}\label{thm:tfunctordiscrete}
	Let $S$ be a discrete $p$-toral group. 
	\begin{enumerate}
		\item Rector's category $\bA_{H^*(S)}$ is equivalent to the category $\cal{F}_S(S)^e$, the full subcategory of elementary abelian subgroups in the fusion category $\cal{F}_S(S)$ of $S$. 
		\item (Gonzalez) For any elementary abelian $p$-subgroup $E < S$ we have
		\[
T_E(H_S^*;\res_{S,E}^*) \cong H^*_{C_S(E)}. 
		\]
	\end{enumerate}
\end{prop}
\begin{proof}
	(1) is shown in the proof of Theorem 5.1 of \cite{bchv}, while (2) is shown in Step 1 of the proof of \cite[Lemma 5.1]{Gonzalez2016Finite} as a consequence of \cite[Proposition 3.4.4]{lannes_ihes}. 	
\end{proof}
We thus deduce the following from \Cref{thm:duflotretract,thm:tfunctordiscrete,thm:cohomlogyplocal}. 
\begin{thm}\label{thm:depthplocalcompact}
Let $\cal{G}= (S,\cal{F})$ be a $p$-local compact group, $C_S$ the maximal central elementary abelian $p$-subgroup of $S$, and $C \le C_S$ a central elementary abelian $p$-subgroup of $S$.  
	\begin{enumerate}
	 	\item If $x_1,\ldots,x_n \in \frak m$ is a sequence of homogeneous elements such that the restriction of $x_1,\ldots,x_n$ form a regular sequence in $H_C^*$, then $x_1,\ldots,x_n$ is a regular sequence in $H_{\cG}^*$. 
	 	\item The depth of $H_{\cG}^*$ is greater than or equal to the rank of $C_S$. In particular, the depth of $H_{\cG}^* \ge 1$. 
	 \end{enumerate} 	
\end{thm}
We now move on to Carlson's theorem. Thus, let $\cG = (S,\cF)$ be a $p$-local compact group, and $E \le S$ a subgroup of $S$. We will always write $i_E$ for the inclusion $E \to S$ of a subgroup. We assume that $E$ is fully centralized in $\cF$ in the sense of \cite[Definition 2.2]{BrotoLeviOliver2007Discrete}. This is not too strong of an assumption, since any subgroup is isomorphic (in $\cF$) to one that is fully $\cF$-centralized.  Given a fully-centralized subgroup $E$, there exists an associated $p$-local compact group $\cC_{\cG}(E) = (C_S(A),\cC_{\cF}(E))$, which is the centralizer $p$-local compact group \cite[Section 1.2]{Gonzalez2016Finite}. On the level of classifying spaces, we have $B\cC_{\cG}(E) \simeq \Map(BE,B\cal{G})_{Bf}$, where $Bf$ is the composite $BE \xr{Bi_E} BS \xr{\theta} B\cG$ \cite{Gonzalez2016Finite}. We will need the following lemma. 
\begin{lem}\label{lem:finiteness}
	Let $i \colon E \to S$ be a morphism from an elementary abelian $p$-group into a discrete $p$-toral group, then $i$ is a monomorphism if and only if $H_E^*$ is a finitely generated $H_S^*$-module via $Bi^* \colon H_{S}^* \to H_E^*$. 
\end{lem}
\begin{proof}
 First, we note that $S$ admits a monomorphism into $U(n)$, for some integer $n$, for example, see the proof of Proposition 2.3 of \cite{JeanneretOsse1997theory}. In fact, the proposition shows that $H_S^*$ is a finitely generated $H_{U(n)}^*$-module. 

 Suppose now that $i \colon E \to S$ is a monomorphism, then the composite $E \hookrightarrow S \hookrightarrow U(n)$ is a monomorphism, and Quillen has shown that this makes $H_E^*$ into a $H_{U(n)}^*$-module \cite[Theorem 2.1]{Quillen1971spectrum}. It follows that $H_E^*$ is a finitely generated $H_S^*$-module. Conversely, if $H_E^*$ is a finitely generated $H_S^*$-module, then it is also a finitely generated $H_{U(n)}^*$-module. By \cite[Corollary 2.4]{Quillen1971spectrum} this is only possible if the composite $E \to S \to U(n)$ is a monomorphism, which forces $E \to S$ to be a monomorphism as well. 
\end{proof}
\begin{prop}
	Let $\cG = (S,\cF)$ be a $p$-local compact group, and let $\cal{F}^e \subset \cF$ denote the full subcategory of $\cF$ whose objects are elementary abelian $p$-subgroups of $S$ which are fully centralized. 
	\begin{enumerate}
		\item There is an equivalence of categories $\cal{F}^e \cong \bA_{H_{\cG}^*}$, which assigned to a fully centralized subgroup $E < S$ the pair $(E,j_E)$, where $j_E \colon H_{\cF}^* \to H_E^*$ is the restriction map.
		\item For each $E \in \cF^e$ there is an isomorphism $T_E(H_{\cF}^*;j_E) \cong H^*_{\cC_{\cF}(E)}$.
	\end{enumerate}
\end{prop}
\begin{proof}
	In order to prove (1), we introduce the category $\bA(B\cG)$, whose objects are pairs $(E,f)$ where $E$ is an elementary abelian $p$-group, and $f \colon BE \to B\cG$ a morphism that makes $H_E^*$ into a finitely generated $H_{\cG}^*$-module, and a morphism $(E,f) \to (E',f')$ is a monomorphism $\phi \colon E \to E'$ such that $f' \circ B\phi \simeq f$. The desired equivalence is  then given as the composite of equivalences 
\[
\xymatrix{
	\cal{F}^e \ar[r]^-{\widetilde B(-)} & \bA(B\cal{G}) \ar[r]^-{H^*(-)} & \bA_{H_{\cG}^*. }
}
\]
where $\widetilde B(E) = (BE,\theta\circ B(i_E))$ on objects, and for a morphism $\phi \colon E \to E'$ in $\cal{F}^e$, we define $\widetilde B(\phi) = \phi$.  

The functor $H^*(-)$ is an equivalence by \cite[Theorem 3.1.1]{lannes_ihes}, so it suffices to show the first functor is an equivalence. We first observe that by \Cref{lem:finiteness,thm:cohomlogyplocal}(1) this is a well-defined functor, i.e., $\widetilde B(E) \in \bA(B\cG)$.  Moreover, by \cite[Theorem 6.3(a)]{BrotoLeviOliver2007Discrete} we easily deduce that $\Hom_{\cal{F}}(E,E') = \Hom_{\bA(B\cG)}(\widetilde B(E),\widetilde B(E'))$. In particular, $\widetilde B$ is fully-faithful. 

Finally, we show that $\widetilde B$ is essentially surjective. Indeed, given $(E,f)$ where $f \colon BE \to B\cG$, by \cite[Theorem 6.3(a)]{BrotoLeviOliver2007Discrete} again, there exists a $\rho \colon E \to S$, unique up to isomorphism in $\cF$, such that $f \cong \theta \circ B\rho$. Moreover, we have $(\rho(E),\theta \circ B(i_{\rho(E)})) \in \bA({B\cG})$, and the restriction of $\rho$ induces an isomorphism $(E,f) \to (\rho(E),\theta \circ B(i_{\rho(E)}))$ in $\bA(B\cG)$. 

Note that $\rho(E)$ need not be fully centralized in $\cal{F}$. However, there exists an isomorphism $\psi \in \Hom_{\cF}(\rho(E),E')$ where $E'$ is fully centralized. Applying \cite[Theorem 6.3(a)]{BrotoLeviOliver2007Discrete} once again, we have that $\theta \circ B(i_{\rho(E)}) \cong \theta \circ B(i_{E'} \circ \psi)$. This shows that there is an isomorphism $\psi \colon (\rho(E),\theta \circ B(i_{\rho(E)})) \to (E',\theta 
\circ B(i_{E'}))$ in $\bA(B\cG)$. The latter is just $\widetilde B(E')$, and hence we deduce that $\widetilde B(E') \cong (E,f)$ in $\bA(B\cG)$. 

	Part (2) is due to Gonzalez \cite[Lemma 5.1]{Gonzalez2016Finite}. 
\end{proof}

With this in mind, Carlson's theorem takes the following form. 
\begin{thm}\label{thm:carlsonplocalcompactgroup}
	Let $\cG = (S,\cF)$ be a $p$-local compact group. If $\depth(H_{\cG}^*) \ge s$, then the morphism
	\[
\xymatrix{H_{\cG}^* \ar[r]& \displaystyle\prod_{\stackrel{E \in \cF^e}{\dim E = s}}H_{\cC_{\cG}(E)}^*}
	\]
	is injective. 
\end{thm}
\subsection{Invariant theory}\label{sec:modular}
Let $V$ be a finite-dimensional $\F_p$-vector space, $G$ a finite group where $p$ divides the order of $G$, and $\rho \colon G \to GL_n(\F_p)$ a faithful modular representation. We will denote by $\F[V]$ the graded algebra of polynomial functions on $V$, i.e., the symmetric algebra on the dual $V^*$. Note that $\F[-]$ forms a contravariant functor. By placing all generators in degree 2, this is a graded $\F_p$-algebra with a unique action of the Steenrod algebra,\footnote{For example, when $p = 2$, for a generator $x$ degree 2, we have $Sq^2(x) = x^2$, and $Sq^i(x) = 0$ otherwise.} and $\F[V]$ defines an element of $\cal{K}$. See \cite[Section 5]{DwyerWilkerson1998Kahler} for more discussion on this construction, where it is called the enhanced symmetric algebra on $V$. Finally, the Steenrod operations commute with the action of $G$, and so also act on $\F[V]^G$.  

Let $U \subset V$ be a subspace, then we can define a $\cK$-map $i_{U} \colon \F[V]^G \to \F[V] \to \F[U] \to H_U^*$. The following can be deduced from \cite[Section 5]{DwyerWilkerson1998Kahler} or \cite[Section 10.1]{NeuselSmith2002Invariant}.
\begin{prop}\label{prop:tfunctorinvariant}
Let $V$ be a finite-dimensional $\F_p$-vector space, $G$ a finite group where $p$ divides the order of $G$, and $\rho \colon G \to GL_n(\F_p)$ a modular representation. For a subspace $U \subset V$, we have 
\[
T_U(\F[V]^G;i_{U}) \cong \F[V]^{G_U}
  \]
  where $G_U < G$ is the pointwise stabilizer of $U$, i.e, $G_U = \{ g \in G \mid g \cdot u = u \} $. 
\end{prop}
Before we state our theorem on depth for invariant rings, we observe that we are in the situation of \Cref{thm:duflotretract}. First, for any subgroup $H \le G$, we have a relative transfer map \cite[Section 2.2]{NeuselSmith2002Invariant}
\[
\Tr^G_H \colon \F[V]^H \to F[V]^G
\]
given by
\[
\Tr^G_H(f)(x) = \sum_{gH \in G/H} g(f)(x)
\]
where the sum is taken over a set of left coset representatives of $H$ in $G$. The composite 
\[
\F[V]^G \to \F[V]^H \to \F[V]^G
\]
is given by multiplication by the index of $H$ in $G$. In particular if $P$ is a Sylow $p$-subgroup of $G$, then this provides the splitting required by \Cref{thm:duflotretract} (see also the discussion before Proposition 1.5 of \cite{DwyerWilkerson1992cohomology}). 

We write $V^G = \{ v  \in V \mid g \cdot v = v \}$ for the $G$-invariant subspace $V^G \subseteq V$.
 We deduce the following. 
\begin{thm}\label{thm:modular}
	Let $G$ be a finite group whose order is divisible by $p$, let $P$ be a Sylow $p$-subgroup of $G$, and let $C \subseteq V^P$ be a subspace. 
	\begin{enumerate}
	 	\item If $x_1,\ldots,x_n$ is a sequence in $\F[V]^G$ such that the restrictions of $x_1,\ldots,x_n$ form a regular sequence in $H_C^*$, then $x_1,\ldots,x_n$ form a regular sequence in $\F[V]^G$.
	 	\item The depth of $\F[V]^G$ is at least $\dim_{\F_p}(V^P)$. 
	 \end{enumerate} 
\end{thm}
\begin{rem}
	Note that since $P$ is a $p$-group, $V^P \ne 0$, see \cite[Lemma 4.0.1]{CampbellWehlau2011Modular}, and we deduce the depth of $\F[V]^G$ is at least 1. Moreover, since the representation is faithful, $\dim_{\F_p}(V^P) < n$, consistent with the fact that $\depth(\F[V]^G) \le n$. 

  We also note that the inequality  
	\[
\depth(\F[V]^G) \ge \dim_{\F_p}(V^P) 
	\]
	follows from a stronger result of Ellingsrud and Skjelbred, namely that 
  \[
\depth(\F[V]^G) \ge \min(\dim_{\F_p}(V^P)+2,n).
  \] 
  The first result in \Cref{thm:modular} appears to be new. 
\end{rem}
\begin{rem}
  Implicit in this theorem is the claim that if $C \subseteq V^P$, then $\F[V]^P$ is a $H_C^*$-comodule. We can see this directly: there is a $P$-equivariant multiplication homomorphism $C \times V \to V$, which induces $\F[V]^P \to F[C] \otimes \F[V]^P$. This descends to a homomorphism $\F[V]^P \to H_C^* \otimes \F[V]^P$, and gives $\F[V]^P$ the desired $H_C^*$-comodule structure. 
\end{rem}
One has a version of Carlson's theorem in this setting as well, however this is trival: the maps in the theorem correspond to the maps $\F[V]^G \to \F[V]^{G_U}$ for $U \subset V$, and $G_U$ the pointwise stabilizer of $U$. Since $G_U$ is a subgroup of $G$, these maps are always inclusions, regardless of the depth of $\F[V]^G$. 

\bibliographystyle{alpha}
\bibliography{nilpotence}

\end{document}